\documentclass[12pt]{amsart}
\usepackage{amssymb, amsmath, amsthm, amsfonts, amscd}
\usepackage{xcolor}
\usepackage{datetime}
\usepackage{hyperref}
\hypersetup{
	colorlinks=true,
	linkcolor=[rgb]{0.64, 0.0, 0.0},
	filecolor=magenta,      
	urlcolor=[rgb]{0.0,0.42,0.1},
	citecolor= [rgb]{0.0,0.32,0.89},
}
\usepackage{mathtools}
\usepackage{mathabx}
\usepackage{enumerate}
\usepackage[mathscr]{euscript}
\usepackage{url}
\usepackage{enumitem}
\usepackage{mathtools}
\usepackage{commath}
\definecolor{cobalt}{rgb}{0.0, 0.28, 0.67}
\usepackage[capitalize,nameinlink]{cleveref}
\bibliography{myrefs}

\newtheorem{theorem}{Theorem}[section]
\newtheorem{lemma}[theorem]{Lemma}
\newtheorem{corollary}[theorem]{Corollary}
\newtheorem{proposition}[theorem]{Proposition}
\theoremstyle{definition}
\newtheorem{definition}[theorem]{Definition}
\newtheorem{example}[theorem]{Example}
\newtheorem{remark}[theorem]{Remark}

\newtheorem{question}[theorem]{Question}

\newcommand{\vp}{\varphi}

\def \C{\mathbb{C}}

\newcommand{\clb}{\mathcal{B}}

\newcommand{\clh}{\mathcal{H}}

\newcommand{\cls}{\mathcal{S}}
\newcommand{\clf}{\mathcal{F}}
\newcommand{\clm}{\mathcal{M}}
\newcommand{\cln}{\mathcal{N}}
\newcommand{\gs}{(\mathcal{GS})}
\newcommand{\bh}{\clb(\clh)}
\numberwithin{equation}{section}

\title[Factorization of anti-linear and $C$-normal operators]{Factorization of anti-linear and $C$-normal operators}

\author{Sudip Ranjan Bhuia }
\address{Indian Statistical Institute, Statistics and Mathematics Unit, 8th Mile, Mysore Road, Bangalore, 560059, India}
\email{\textcolor{cobalt}{sudipranjanb@gmail.com}}

\subjclass[2020]{Primary 47A65 ; Secondary 47A68, 47B15}
\keywords{Anti-linear operators, $C$-normal operators, normal operators, Douglas factorization, Polar decomposition, Cartesian decomposition }
\date{\currenttime ;  \today}

\begin{document}
	
	\maketitle
	
\begin{abstract}
	A conjugation $C$ is an anti-linear isometric involution on a complex Hilbert space $\clh$, and $T\in \clb(\clh)$ is conjugate normal if $T^*T = CTT^*C$ holds for some conjugation \(C\). In this paper, we provide a factorization and range inclusion theorem for anti-linear operators, and consequently, establish the polar decomposition for anti-linear operators by applying the Douglas theorem on majorization of Hilbert space operators. Moreover, we present a factorization of $C$-normal operators based on the polar decomposition. Lastly, we study Cartesian decomposition of conjugate normal operators, thereby expanding the results in \cite{RSV_Cart}.
	
\end{abstract}

\section{Introduction}
Throughout the paper, $\clh$ denotes a separable, complex Hilbert space endowed with an inner product $\langle\cdot,\cdot\rangle$  and $\bh$ is the $C^*$-algebra of bounded linear operators on $\clh$.

A map $C:\clh\rightarrow\clh$ is said to be a conjugation if 
\begin{enumerate}
	\item $C(\alpha x+y)=\bar{\alpha}Cx+Cy$ for all $\alpha\in \C$ and for all $x,y\in \clh$,
	\item $C^2=I$, and
	\item $\langle Cx,Cy\rangle=\langle y,x\rangle$ for all $x,y\in \clh$. 
\end{enumerate} 
That is, $C$ is an anti-linear, isometric, involution. Denote $\clb_{a}(\clh)$ by the collection of all anti-linear bounded operators on $\clh$.
\begin{definition}\label{antilinear operator adjoint}
	For $X\in \clb_{a}(\clh)$, there is a unique $X^{\#}\in \clb_{a}(\clh)$ called the anti-linear adjoint of $X$, if it satisfies the following relation
	\begin{equation}
		\langle Xx,y \rangle =\overline{\langle x, X^{\#}y \rangle}=\langle  X^{\#}y,x \rangle
	\end{equation}
	for all $x,y \in \mathcal{H}$.
\end{definition}\cite{Ptak}
The anti-linear operator $X$ is called anti-linear self-adjoint if $X^{\#}=X$, and anti-linearly normal if $X^{\#}X=XX^{\#}$. Conjugations are examples of anti-linear self-adjoint operators.

The study of conjugations has its roots in physics and is also essential in operator theory and operator algebras.
Conjugations and the extension theory for unbounded symmetric operators are closely related concepts. A survey article titled ``Mathematical and physical aspects of complex symmetric operators'' by Garcia, Prodan, and Putinar \cite{Garcia:Prodan:Putinar} is referenced as a source for further classical remarks on conjugations.

The study of complex symmetric operators involves conjugations in a significant way. Its structure is described by the term "complex symmetric." An operator $T$ is said to complex symmetric if and only if it has a symmetric matrix representation with respect to some orthonormal basis of $\clh$. Here is the formal definition.
\begin{definition}\label{complex symmetric}
	Let $T\in \clb(\clh)$ and $C$ be a conjugation on $\clh$. Then $T$ is said to be $C$-symmetric if $T^*=CTC$. Equivalently, $T^*C=CT$, or $TC=CT^*$, and $T$ is said to be a $C$-skew-symmetric if $CT^*C=-T$. 
\end{definition}
Thus a bounded linear operator $T$ is complex symmetric if there exists a conjugation $C$ with respect to which it is $C$-symmetric.
Let $\cls_{C}(\clh)=\{T\in \clb(\clh) : T^*=CTC \}$ be the space of all $C$-symmetric operators. $T$ is
complex symmetric provided $T\in \cls_{C}(\clh)$ for some conjugation $C$. Garcia, and Putinar \cite{Garcia:Putinar-1,Garcia:Putinar-2,Garcia:Prodan:Putinar,Garcia:AAEPRI} have initiated the study of complex symmetric operators in a systematic way. They have shown that the class is remarkably extensive and contains the normal operators, the Hankel operators, the compressed Toeplitz operators, several integral operators. The complex symmetric weighted shift is characterized by S. Zhu and C.G. Li in \cite{Zhu:Li: symmetric weighted shift}. For more details on complex skew-symmetric operators readers are referred to \cite{Li: Zhu, Sen Zhu: Skew weighted shift}

In their recent paper \cite{Ptak}, by providing an intriguing generalization, Ptak, Simik, and Wicher provide a comprehensive framework for the study of complex symmetric and skew-symmetric operators. 
	\begin{definition}\label{conjugate normal}
	An operator $T\in \bh$ is said to be conjugate normal if $C|T|C=|T^*|$ for some conjugation $C$ (in this case, $T$ is said to be $C$-normal).
\end{definition}

Hence, it is evident that when $T$ is $C$-normal, the equality $C\ker T=\ker T^*$ holds implying that any injective $C$-normal operator possesses a dense range (in literature such operators are know as quasi-affine operators). It is very easy to check that complex symmetric and skew-symmetric operators are conjugate normal.
By $T\in \gs$, we mean $T$ is conjugate normal for some conjugation $C$. For a fixed conjugation, we denote $\cln_C(\clh)$ by the set of all $C$-normal operators on $\clh$. There has been a growing interest in the study of conjugate normal operators and became recent trends in operator theory. Very recently Ramesh, Sudip, and Venku Naidu \cite{RSV_Cart,RSVcptrpn} have studied Cartesian decomposition of conjugate normal operators, spectral representation of compact conjugate normal operators, and they have introduced the eigenvalue problem for conjugate normal operators. For more example and properties of conjugate normal operators, we refer \cite{Bhuia,Ko et. al,Wang:Zhu}.

In \cite{Liu:Zhu-1,Liu:Zhu-2,Wang:Xie:Yan:Zhu}, authors have considered the interpolation problem for conjugations, that is, they have investigated the pairs of operator $(P,Q)$ satisfying $CPC=Q$ for some conjugation $C$ whenever $P,Q$ are orthogonal projections, simultaneously diagonal operators, and commuting self-adjoint operators.

 In a very recent paper \cite{Liu:Zhu-1}, Liu, Shi, Wang, and Zhu gave a complete characterization of conjugate normal partial isometries by solving the interpolation problem for conjugation (\cite[Problem 1.2]{Liu:Zhu-2}) with pair of partial isometries.

In \cite{Liu:Zhu-2}, authors have solved the interpolation problem for pair $(P,Q)$ of simultaneously diagonal operators as a consequence, they have obtained a general characterization of conjugate normal weighted shifts. Additionally, they have noted that the problem of figuring out if an operator is conjugate normal is really a specific instance of the interpolation problem for conjugation. 

In \cite{Wang:Xie:Yan:Zhu}, authors have solved the interpolation problem for conjugation for a pair $(P,Q)$ of self-adjoint operators with $PQ=QP$ and as an application they have obtained some results on factorization of unitary operators. Also, authors have shown that if $T\in \bh$ and $T^2=0$, then $T$ is conjugate normal as corollary to their main results.

It follows that the interpolation problem for conjugations plays a significant part in characterizing conjugate normal operators.

In this paper, our main aim is to give the factorization of conjugate normal operators based on the polar decomposition theorem (cf. \cref{cpolar}).

Organization of this paper: In Section $2$, preliminaries results are discussed which will be used in the subsequent sections. Using the Douglas famous theorem, we give factorization and range inclusion theorem for anti-linear operators in Section $3$. In Section $4$, we give factorization of conjugate (or, anti-linear) normal operators using multiplicity theory of normal operators. In Section $5$, we give a different proof of Cartesian decomposition of conjugate normal operators and its characterization. Alongside, we give an application of Cartesian decomposition of compact conjugate normal operators to singular value inequalities.

\section{Preliminaries}
For any bounded linear operator $A$ between linear spaces, the range and the null space of $A$ are denoted by $\text{ran}(A)$ and $\ker(A)$, respectively.

In 1966, R. G. Douglas established the following celebrated assertion, known as the Douglas theorem or Douglas majorization theorem.
\begin{theorem}[Douglas Theorem]\cite[Theorem 1]{Douglas Lemma}
	 If $A, B \in \clb(\clh)$, then the following statements are equivalent:
	\begin{enumerate}
		\item[{\rm (i)}] $\text{ran}(A) \subseteq \text{ran}(B)$;
		\item[{\rm (ii)}] $A=BC$ for some $C \in \clb(\clh)$;
		\item[{\rm (iii)}]$AA^* \leq k^2 BB^*$ for some $k\geq 0$.
	\end{enumerate}
	Moreover, if {\rm (i)}, {\rm (ii)}, and {\rm (iii)} are valid, then there exists a unique operator $C$ so that
	\begin{enumerate}
		\item[{\rm (a)}] $\|C\|^2=\inf\{\lambda|AA^*\leq \lambda BB^*\}$;
		\item[{\rm (b)}] $\ker(A) = \ker(C)$;
		\item[{\rm (c)}] $\text{ran}(C) \subseteq \overline{\text{ran}}(B^{*})$.
	\end{enumerate}
\end{theorem}

The following result gives the description of anti-linear operators that commute with normal operator. 
	\begin{lemma}\cite[Lemma 3.3]{Gilbreath:Wogen}\label{anti-commut-normal}
		Suppose that $K$ is a conjugation and that the normal operator $N$ lies
		in $\cls_{K}(\clh)$. Then $\{ A\in \clb_{a}(\clh):AN=NA \}=\{TK:T\in \clb(\clh)\,\,\text{and}\,\, TN^*=NT  \}$.
	\end{lemma}
A similar description as above can be concluded about anti-linear operators that commute with complex symmetric operator. That is, if an anti-linear operator $A$ commute with a $C$-symmetric operator $T$, then $A=SC$ and $ST^*=TS$ for some $S\in \bh$.


	The following are equivalent conditions for a $C$-normal operator. 
	\begin{theorem}\cite[Theorem 2.3]{Ptak}\label{ptak_charac}
		Let $C$ be a conjugation on $\clh$, and let $T\in \clb(\clh)$. Then the following are equivalent:
		\begin{enumerate}
			\item $T$ is C-normal,
			\item $T^*$ is $C$-normal.
			\item $CTC$ is $C$-normal.
			\item $CT^*C$ is $C$-normal.
			\item $CT^*T=TT^*C$.
			\item $CT(CT)^{\#}=(CT)^{\#}CT$.
			\item\label{norm condition C-normal} $\left\|TCx\right\|=\left\|T^*x\right\|$ for $x\in \clh$.
			\item $\left\|T^*Cx\right\|=\left\|Tx\right\|$ for $x\in \clh$.
			\item $T_{+}:=\frac{1}{2}(CT+T^*C)$ and $T_{-}:=\frac{1}{2}(CT-T^*C)$ commute.
			\item $T_{+}:=\frac{1}{2}(TC+CT^*)$ and $T_{-}:=\frac{1}{2}(TC-CT^*)$ commute.
		\end{enumerate}	
	\end{theorem}

	\begin{lemma}\cite[Lemma 2.4]{Ptak}
		Let $C$ be a conjugation on $\clh$. If $T\in \mathcal{B}(\clh)$ is $C$-normal, then $T_{L}=CTCT$ and $T_{R}=TCTC$ are normal operators.
	\end{lemma}	
	\begin{remark}\label{c-symmetric from c-normal}
		The following gives a way to construct $C$-symmetric operator and $C$-skew-symmetric operator given a $C$-normal operator.
		\begin{enumerate}
			\item If $T$ is $C$-normal, then $S_1=T^*T-TT^*$ is $C$-skew-symmetric.
			\item If $T$ is $C$-normal, then $S_2=T^*T+TT^*$ is $C$-symmetric.\\
		\end{enumerate}
	\end{remark}

\begin{theorem}\cite[Theorem 1.6]{Wang:Zhu}\label{cpolar}
	Let $C$ be a conjugation on $\mathcal{H}$ and $T \in \mathcal{B(H)}$. Then the following are
	equivalent:
	\begin{enumerate}
		\item $T$ is $C$-normal.
		\item $T = CJ|T|$ for some partial anti-unitary operator $J$ supported on $\text{ran}|T|$ with
		$J|T| = |T|J$.
		\item $T =C\tilde{J}|T| $ for some anti-unitary operator $\tilde{J}$ with $\tilde{J}|T|=|T|\tilde{J}$.
	\end{enumerate}
\end{theorem}
The following is a simple consequence of the above theorem which is parallel to the case of complex symmetric operators (cf. \cite[Theorem 3]{garciaCSPI}).
\begin{remark}
	If $T$ is a C-normal partial isometry, then $T$ can be written as a product of a $C$-normal unitary operator $U$ and a projection $P$ that is, $T=UP$.
\end{remark}
\begin{proposition}\cite[Proposition 4.3]{Wang:Zhu}\label{cnormaldirectsum}
	Let $\clm$ be a subspace of $\clh$ and $T\in \clb(\clh)$ with 
	\begin{align}
	\begin{pmatrix}
	A&0\\
	0&B
	\end{pmatrix},
	\end{align}
	where $A:\clm\rightarrow \clm$ is a diagonal operator. Then $T\in \gs$ if and only if $B:\clm^{\perp}\rightarrow \clm^{\perp}\in\gs$.
\end{proposition}

If $\mu$ is a regular Borel measure on $\C$ with compact support, define $N_\mu$ on $L^2(\mu)$ by 
\begin{equation}
N_\mu f=zf
\end{equation} for each $f\in L^2(\mu)$. It is easy to check
that $N^*_\mu f=\bar{z}f$ and, hence, $N_\mu$ is normal. The the following are true:
\begin{enumerate}
	\item $\sigma(N_\mu)=$ support of $\mu$.
	\item  If, for a bounded Borel function $\vp$, we define $M_\vp$ on $L^2(\mu)$ by $M_\vp f=\vp f$,
	then $\vp(N_\mu)=M_\vp$.
	\item If $E$ is the spectral measure for $N_\mu$, then $E(\Omega)=M_{\chi_\Omega}$.
\end{enumerate}
For $A,B\in \bh$, we say $A$ is unitarily equivalent to $B$, that is, $A\cong B$ if there exists an unitary operators $U$ such that $U^*AU=B$.
\begin{theorem}\cite[Theorem 10.16]{Conway:Book:Functional analysis}\label{Dsum_Normal}
	If $N$ is a normal operator, then there are mutually singular
	measures $\mu_\infty,\mu_1,\mu_2,\dots$(some of which may be zero) such that
	\[N\cong N^{(\infty)}_{\mu_\infty}\oplus N_{\mu_1}\oplus N^{(2)}_{\mu_2}\oplus N^{(3)}_{\mu_3}\oplus\cdots .\]
\end{theorem}
If $\clh=\displaystyle\oplus_{n\geq 1} \clh_n$ and $A=\displaystyle\oplus_{n\geq 1} A_n$, where $A_n\in \clb(\clh_n)$ . It is well know that $A$ is bounded inf and only if $\sup_{n}\|A_n\|<\infty$ and in this case $\|A\|=\sup_{n}\|A_n\|$. Now suppose that $B\in \bh$, then $B$ has a operator matrix representation $B=[B_{ij}]$, where $B_{ij}\in \clb(\clh_j,\clh_i)$. Let $T\in \bh$. Then $\{T\}^\prime=\{S\in \bh : ST=TS \}$ is called the commutant of $T$. The following result concerns about the the commutant of a direct sum of operators. 
\begin{proposition}\cite[Proposition 6.1]{Conway:Book:Functional analysis}\label{direct sum commut prop}
	The following are true:
	\begin{enumerate}
		\item If $A=\oplus_{n\geq 1} A_n$ is a bounded operator on $\clh=\oplus_{n\geq 1} \clh_n$ and $B=[B_{ij}]\in \bh$, then $AB=BA$ if and only if $B_{ij}A_j=A_iB_{ij}$ for all $i,j$.
		\item If $B=[B_{ij}]\in \clb(\clh^{(n)})$, then $BA^{(n)}=BA^{(n)}$ if and only if $B_{ij}A=AB_{ij}$ for all $i,j$, where $\clh^{(n)}$ and $A^{(n)}$ are the direct sum of $n$ copies of $\clh$ and $A$,respectively.
	\end{enumerate}
	
\end{proposition}
It is well known fact that if $\mu$ is a compactly supported measure on $\C$, then
\[ \{N_\mu\}^{\prime}=\{M_\vp :\vp\in L^\infty(\mu)\} .\]
\begin{proposition}\cite[Proposition 3.5]{Gilbreath:Wogen}\label{conju commute positive}
	$J$ is a conjugation commuting with $N_\mu^{(n)}$ if and only if $J$ has the form $J=\left(M_{\vp_{i j}}\right) S_\mu^{(n)}$, where $\left(M_{\vp_{i j}}\right)$ is a unitary operator matrix such that $\vp_{i j}=\vp_{j i} \in L^{\infty}(\mu)$, for all $i,j$.
\end{proposition}

\section{Factorization and range inclusion theorem for anti-linear operators}
This section is devoted to application of Douglas theorem. Here we prove the range inclusion and factorization theorem for anti-linear operators and also, we prove the polar decomposition of anti-linear operators.
\begin{theorem}
	Let $S, T\in \clb_{a}(\clh)$, and $C$ be a conjugation on $\clh$. Then the fallowing are equivalent:
	\begin{enumerate} 
		
		\item\label{AD1}  $\text{ran}(T) \subseteq \text{ran}(S)$.
		
		\item\label{AD2}  $T T^{\#} \leq S S^{\#}$.
		
		\item\label{AD3} $T=S R$, where $R=CDC$ for some $D \in \clb(\clh)$.
	\end{enumerate}
	Moreover, for such an operator $R$, we have $\ker R=\ker T$ and $\overline{\text{ran}}(R)=(\ker S)^\perp. $
\end{theorem}
\begin{proof}
	(\ref{AD1}) $\implies$(\ref{AD3})
	Let $C$ be a conjugation on $\clh$. Set $A=T C$ and $B=S C$. Then $A$ and $B$ are linear operators. Since $C \clh=\clh$, we have
	
	\begin{equation*}
		\text{ran}(A)=\text{ran}(T C)=\text{ran}(T) \subseteq \text{ran}(S)=\text{ran}(S C)=\text{ran}(B).
	\end{equation*}
	
	Then by Douglas theorem, there is a $ D \in \clb(\clh)$ such that
	
	$$
	\begin{aligned}
		& A=B D \\
		\Rightarrow T C & =S C D \\
		\Rightarrow T & =S C D C =S R, \text { where } R=C D C .
	\end{aligned}
	$$
	
	Note that $R\in \clb(\clh)$.\\
	(\ref{AD3})$\implies$ (\ref{AD1})
	Using the fact that $C\clh=\clh$, we get $T\clh=S R \clh \subseteq S \clh $ and this implies $ R(T) \subseteq R(S)$.\\
	(\ref{AD2}) $\implies$ (\ref{AD3}) Observe that $A A^{*}=T C C T^{\#}=T T^{\#}$ and $B B^{*}=S S^{\#}$. Therefore, $T T^{\#} \leq S S^{\#}$ implies $A A^{*} \leq B B^{*} $. By Douglas theorem, $ A=BD$ and hence $T=S R$.\\
	(\ref{AD3}) $\implies$ (\ref{AD2}) If $T=SR$, then
	\begin{equation*}
		\begin{split}
			S S^{\#}-T T^{\#}  &=S S^{\#}-S C D C C D^{*} C S^{\#} \\
			&=S S^{\#}-S C D D^{*} C S^{\#} 
			=S C\left(I-D D^{*}\right) C S^{\#} 
			=S C\left(I-D D^{*}\right)(S C)^{*}	.
		\end{split}
	\end{equation*}

	Thus $T T^{\#} \leq S S^{\# } $.

	Since $\ker D=\ker A$ and $\overline{\text{ran}}(D)=(\ker B)^{\perp}$, we have

	\begin{equation*}
		\ker D=\ker (TC)=C\ker T\quad\text{and}\quad \overline{\text{ran}}(D)=(\ker (SC))^\perp=(C\ker S)^\perp=C(\ker S)^\perp.
	\end{equation*}
	
	Thus $$\ker R= \ker (CDC)=\ker (DC)=C\ker D=C\ker(TC)=\ker T, $$ and	
	$$\overline{\text{ran}}(R)=\overline{\text{ran}}(CDC)=(\ker(CD^*C) )^\perp=(\ker(D^*C) )^\perp=C\overline{\text{ran}}(D)=(\ker S)^\perp. $$
	This completes the proof.
\end{proof}
\begin{theorem}
	Let $S,T\in\clb_a(\clh)$ such that $S^{\#} S=T^{\#} T$. Then there exists a partial isometry $D\in \clb(\clh)$  with initial space $\overline{\text{ran}}(T)$ and final space $\overline{\text{ran}}(S)$ such that $S=DT$.
\end{theorem}
\begin{proof}
Let $C$ be a conjugation. Then	$ S^{\#} S=T^{\#} T$ implies	$ C S^{\#} S C=C T^{\#} T C$. Now set $A=C S^{*}$ and $B=C T^{*}$. Then $A A^{*}=B B^{*}$. By Douglas Theorem,
	
	$$A=B D^{*},$$ where $D$ is a linear partial isometry with initial Space $\overline{\text{ran}}\left(B^{*}\right)$ and final space $\overline{\text{ran}}\left(A^{*}\right)$.
	
	This implies
	
	\begin{equation*}
		\begin{split}
				C S^{*} & =C T^{\#} D^{*} \\
			\quad S^{\#} & =T^{\#} D^{*} \\
			S & =D T.
		\end{split}
	\end{equation*}
	
	Since $\overline{\text{ran}}\left(B^{*}\right)=\ker(B)^{\perp}=\ker\left(C T^{*}\right)^{\perp}=\ker\left(T^{*}\right)^{\perp}=\overline{\text{ran}}\left(T\right)$ and $\overline{\text{ran}}\left(A^{*}\right)=\overline{\text{ran}}(S).$ The operator $D$ is linear partial isometry with initial space $\overline{\text{ran}}(T)$ and final space $\overline{\text{ran}}(S)$.
\end{proof}
Next we give the polar decomposition of anti-linear operators which is also available in \cite{RSV:C-Polar} but here we establish this as a consequence of Douglas theorem. 

\begin{theorem}
	Let $A\in \clb_a(\clh)$ and $C$ be a conjugation. Then there exists an unique anti-linear partial isometry $J$ such that $A=J|A|$ such that $\ker J=\ker A$.
\end{theorem}
\begin{proof}
	Since $A$ is an anti-linear operator and $C$ is a conjugation, the operator $CA$ is bounded linear. Then note that $|CA||CA|=(CA)^*(CA)$ and $|A|=|CA|$. Set $\tilde{A}=(CA)^*$ and $\tilde{B}=|A|$ and then by applying Douglas theorem, we get
	\begin{equation*}
		\begin{split}
			\tilde{A}&=\tilde{B}D\quad\text{which implies}\\
			A&=CD^*|A|\\
			A&=J|A|,
		\end{split}
	\end{equation*}
	where $	J=CD^*$	for some $D\in \clb(\clh)$ with $\ker D=\ker \tilde{A}$ and $\overline{\text{ran}}D=(\ker \tilde{B})^\perp$. Therefore,
	\begin{equation*}
		\ker J=\ker (CD^*)=\ker D^*=(\text{ran}D)^\perp=\ker\tilde{B}=\ker |A|=\ker A.\end{equation*}
	Also, $J$ is anti-linear and isometry on  $\overline{\text{ran}}(|A|)=(\ker J)^\perp.$
\end{proof}

Below we give a proof of polar decomposition of $C$-normal operators (cf. Theorem \ref{cpolar}) by using Douglas theorem.
\begin{theorem}
	Let $T\in \clb(\clh)$ and $C$ be a conjugation. Then $T$ is $C$-normal operator if and only if there exists an partial anti-unitary operator $J$ such that $T=CJ|T|$ with $J|T|=|T|J$.
\end{theorem}
\begin{proof}
	Let $T\in \clb(\clh)$ be $C$-normal operator that is, $CT^*TC=TT^*$ equivalently, $C|T|C=|T^*|$. Then by setting $A=CT^*C$ and $B=|T^*|$, we have $AA^*=BB^*$. Now apply Douglas theorem, we get $A=BD^*$ for some $D:\overline{\text{ran}}(B^*)\rightarrow\overline{\text{ran}}(A^*)$ such that $DB^*x=A^*x$, and $Dx=0$ for all $x\in \overline{\text{ran}}(B^*)^\perp=\ker B$. Thus we have
	\begin{equation*}
		\begin{split}
			CT^*C&=|T^*|D^*\\
			T&=CD|T^*|C\\
			T&=CDC|T|.
		\end{split}
	\end{equation*}
	By defininig $J=DC$, we get $T=CJ|T|$. Note that
	\begin{equation*}
		\begin{split}
			\ker J=\ker DC=C\ker D&=C\ker B=	C\ker |T^*|\\&=\ker (|T^*|C)=\ker (C|T|)=\ker |T|=\overline{\text{ran}}(|T|)^\perp,
		\end{split}
	\end{equation*}
	and
	\begin{equation*}
		\begin{split}
			\overline{\text{ran}}(J)=\overline{\text{ran}}(DC)=\overline{\text{ran}}(D)&=\overline{\text{ran}}(CTC)=C\overline{\text{ran}}(T)=C(\ker T^*)^\perp\\&=(\ker |T^*|C)^\perp=(\ker C|T|)^\perp=\overline{\text{ran}}(|T|)
		\end{split}
	\end{equation*}
	Since $\overline{\text{ran}}(B^*)=\overline{\text{ran}}(|T^*|)=C\overline{\text{ran}}(|T|)$, and $\overline{\text{ran}}(A^*)=\overline{\text{ran}}(|T|)$,
	it is evident that $J=DC$ is anti-linear unitary operator supported on $\overline{\text{ran}}(|T|)$.
	
	Also,
	\begin{equation*}
		|T^*|^2=TT^*=CJ|T||T|J^\#C=CJ|T|J^\#CCJ|T|J^\#C
	\end{equation*}
This implies $|T^*|=CJ|T|J^\#C$ and consequently, $J|T|=|T|J$.

The converse follows trivially.
\end{proof}

\section{Factorization of conjugate normal operators}
Let us begin this section with a few insightful observations.
These findings are significant and fit well within the context of conjugate normal operators, even if we have not used them to achieve the primary objective of this section.
\begin{lemma}\label{compression C-normal}
	Let $T\in \clb(\clh)$. If $\clm$ is a reducing subspace of $T$ and $T|_\clm=0$, then $T\in \gs$ if and only if $T|_{\clm^\perp}\in \gs$.
\end{lemma}
\begin{proof}
The proof follows from \cref{cnormaldirectsum}.
\end{proof}
\begin{corollary}
	Let $T\in \clb(\clh)$ be normal. Then $T\in \gs$  if and only if $T|_{(\ker T)^\perp}\in \gs$.
\end{corollary}
\begin{proof}
	Since $T$ is normal, $\ker T$ is a reducing subspace for $T$ (cf. \cite[Proposition 5.6]{Conway:Book:Functional analysis}). Hence by \cref{compression C-normal}, the conclusion follows.
\end{proof}
Below we are giving a different proof of the above result.
\begin{lemma}
Let $C$ be a conjugation on a separable complex Hilbert space $\clh$. Suppose that $T\in \clb (\clh)$ is a $C$-normal operator. If $T$ is normal, then $T|_{(\ker T)^\perp}$ is $C_2$-normal, where $C_2:(\ker T)^\perp\rightarrow(\ker T)^\perp$ is a conjugation.	
\end{lemma} 
\begin{proof}
	Since $T$ is normal, we have $\ker T=\ker T^*$. Write $\clh=\ker T\oplus(\ker T)^\perp$. Then $T$ can be written as $T= \begin{bmatrix}
	0&0\\
	0&A
	\end{bmatrix}$, where $A:(\ker T)^\perp\rightarrow(\ker T)^\perp$.
	
	Now $T$ is $C$-normal which implies $C\ker T=\ker T^*=\ker T$. Therefore, $C$ can be written as $C=\begin{bmatrix}
	C_1&0\\
	0&C_2
	\end{bmatrix}$, where $C_1:\ker T\rightarrow\ker T$ and $C_2:(\ker T)^\perp\rightarrow(\ker T)^\perp$. Note that $C_2$ is also a conjugation on $(\ker T)^\perp$ and $CT^*TC=TT^*$ implies that $C_2A^*AC_2=AA^*$ which implies $A:=T|_{(\ker T)^\perp}$ is $C_2$-normal.
\end{proof}

\begin{proposition}
Let $C$ be a conjugation, and $T\in \bh$ be $C$-normal. Suppose that $S_2=T^*T+TT^*$, then $S_2\cong N\oplus N^*\oplus A$, where $N$ is normal operator and $A$ is self-adjoint operator.
\end{proposition}
\begin{proof}
Since $T$ is $C$-normal, the operator $S_2$ satisfies $CS_2C=S_2$. Hence, by using \cite[Theorem 5]{Wang:Xie:Yan:Zhu}, the conclusion follows.
\end{proof}

Also, if $T$ is $C$-normal for some conjugation $C$, then $S_1=T^*T-TT^*$ satisfies $CS_1C=-S_1$, or equivalently, $CS_1+S_1C=0$. Therefore, this observation allow us to ask the following:
\begin{question}\label{CTC=-T}
	\emph{Characterize all normal operators $T$ satisfying $CTC=-T$ for some conjugation $C$.}
\end{question}

In \cite[Theorem 5]{Wang:Xie:Yan:Zhu}, authors gave a concrete description of normal operators commuting with conjugations (that is, $CTC=T$). In the forthcoming result, we characterize normal operators $T$ that satisfies $CTC=-T$ for some conjugation $C$.

 The results in \cite[Theorem 2.2]{Li: Zhu} and \cite[Lemma 6]{Wang:Xie:Yan:Zhu} are quite helpful in this occasion of solving the \cref{CTC=-T}. In order to do so, let us begin with some useful notations. Denote $\Sigma_1=\{\alpha\in \C: \Re(\alpha)>0   \}$ and $\Sigma_2=\{\alpha\in \C: \Re(\alpha)=0   \}$, where $\Re(\alpha)$ is the real part of a complex number $\alpha$. Given a subset $\Gamma$ of the complex plane $\C$, we denote $\Gamma^*=\{\bar{\alpha}:\alpha\in \C  \}$ and $E_T$ to signify the spectral measure associated with the operator $T$.
\begin{lemma}\label{key-lemma}
	Let $T\in \bh$ be normal and $C$ be a conjugation such that $CTC=-T$, or equivalently, $CT+TC=0$. Then $C\text{ran}E_{T}(\Delta)=\text{ran}E_{T}(-\Delta^*)$, where $\Delta=\Sigma_1\cap\sigma(T)$.
\end{lemma}
\begin{proof}
	Since $CTC=-T$, we have $\sigma(T)=-\sigma(T)^*$. Let $f\in C(\sigma(T))$, the space of all continuous functions defined on $\sigma(T)$, define $f^*(z)=\overline{f(-\bar{z})}$ for all $z\in \sigma(T)$. Then the function $f^*\in C(\sigma(T))$. For $g(z)=\alpha z^n\bar{z}^m$, it follows that
	\[ Cg(T)C=C(\alpha T^n(T^*)^m)C=\bar{\alpha}(-T)^n(-T^*)^m=g^*(T).\]
	Since all bivariate polynomials in $z$ and $\bar{z}$ are uniformly dense in $C(\sigma(T))$, it follows that
	\begin{equation}
		Cf(T)C=f^*(T), \quad\text{for all} \,f\in C(\sigma(T)).
	\end{equation} 
	
	Let $\chi_{\Delta}$ be the characteristic function of $\Delta$ belongs to $L^\infty(\mu)$, where $\mu$ is a scalar-valued spectral measure for $T$. Thus there exits a sequence $\{f_n\}$ in $C(\sigma(T))$ such that $f_n\rightarrow\chi_{\Delta}$ and $f^*_n\rightarrow\chi_{-\Delta^*}$ in the weak$^*$-topology of $L^\infty(\mu)$. Thus by functional calculus for normal operators (cf. \cite[Theorem 8.10]{Conway:Book:Functional analysis}), we get $f_n(T)\rightarrow\chi_{\Delta}(T)$ and $f^*_n(T)\rightarrow\chi_{-\Delta^*}(T)$ in the weak operator topology. Therefore, we obtain $C\chi_{\Delta}(T)C=\chi_{-\Delta^*}(T)$, that is, $CE_T(\Delta)C=E_T(-\Delta^*)$, and hence the desired conclusion follows.
\end{proof}
\begin{theorem}
Let $T\in \bh$ be a normal operator. Then there exists a conjugation $C$ such that $CTC=-T$ if and only if $T\cong T_1\oplus (-T_1^*) \oplus (iT_3)$, where $T_1$ is normal and $T_3$ is self-adjoint. 
\end{theorem}
\begin{proof}
Denote $\Delta_1=\Sigma_1\cap \sigma(T)$ and $\Delta_2=\Sigma_2\cap \sigma(T)$. Then $\sigma(T)=\Delta_1\cup (-\Delta_1^*)\cup \Delta_2$. Denote $\clh_1=\text{ran} E_T(\Delta_1)$, $\clh_2=\text{ran}E_T(-\Delta_1^*)$, and $\clh_3=\text{ran}E_{T}(\Delta_2)$. Then $\clh=\clh_1\oplus\clh_2\oplus\clh_3$. Since $T$ is normal, the subspaces $\clh_1,\clh_2,\clh_3$ reduces $T$. Then $T=T_1\oplus T_2\oplus T_3$, where $T_i =P_{\clh_i}T|_{\clh_i}$. By \cref{key-lemma}, we have $C\clh_1=\clh_2$ and $C\clh_2=\clh_1$. Therefore, we have 
\[C=\begin{bmatrix}
0&D^{-1}&0\\
D&0&0\\
0&0&C_3
\end{bmatrix},\]
where $D:\clh_1\rightarrow\clh_2$ is an anti-linear, invertible, isometry, and $C_3:\clh_3\rightarrow\clh_3$ is a conjugation. Since $T$ satisfies $CTC=-T$, we have $T_2=-DT_1D^{-1}$. Again, $T_1$ is normal, so there exists a conjugation $C_1$ on $\clh_1$ such that $C_1T_1C_1=T_1^*$, and hence we have $T_2=-DC_1T_1^*C_1D^{-1}$. Observe that $DC_1$ is an unitary operator. Thus $T_2\cong -T_1^*$ (unitarily equivalent).
\end{proof}
The following result follows very easily from the above theorem and \cref{c-symmetric from c-normal}.
\begin{corollary}
Let $T\in \bh$ be a $C$-normal operator. Then  $S_1=T^*T-TT^*\cong N\oplus (-N^*) \oplus (iS)$, where $N$ is normal operator and $S$ is self-adjoint operator.
\end{corollary}

	\begin{theorem}\label{commute:spectral measure and conju}
		Let $C$ be a conjugation, and $T\in \clb(\clh)$ be $C$-normal. Then we have the following:\begin{enumerate}
			\item\label{normal:C-normal:spectral measure} If $T$ is normal, then the conjugation $C$ commutes with spectral measure of $T^*T$. That is, if $E$ is the projection-valued spectral measure of $T^*T$, then 
			\begin{equation*}
			C E(G)=E(G)C
			\end{equation*}
			for every Borel subset $G$ of $[0,\infty)$. In particular, $C$ reduces the range of $E(G)$.
			\item\label{spectral measure: CT} The anti-linear operator $CT$ commutes with projection-valued spectral measure of $T^*T$.
		\end{enumerate}
	\end{theorem}
 
\begin{proof}
{\bf Proof of (\ref{normal:C-normal:spectral measure})}	Since $T$ is $C$-normal as well as normal, we have 
	\begin{equation*}
	CT^*T=TT^*C=T^*TC
	\end{equation*}
	which shows that $C$ commutes with $T^*T$. Then $C$ commutes with $f(T^*T)$ for every $f\in C(\sigma(T^*T))$. Let $G$ be a nonempty open subset of $\sigma(T^*T)$. Let $\{f_n\}\subseteq C(\sigma(T^*T))$ be such that $0\leq f_n\leq \chi_G$ and $f_n\nearrow \chi_G$. Then
	\begin{equation}
	\begin{split}
	\overline{\langle CE(G)g,h\rangle}&=\langle E(G)g,Ch\rangle\\
	&=E_{g,Ch}(G)\\
	&=\lim \int f_n(\lambda)dE_{g,Ch}(\lambda)\\
	&=\lim\langle f_n(T^*T)g,Ch\rangle\\
	&=\lim\langle h,Cf_n(T^*T)g\rangle\\
		&=\lim\langle h,f_n(T^*T)Cg\rangle,
	\end{split}
	\end{equation}
	and it follows that
	\begin{equation}
	\begin{split}
	\langle CE(G)g,h\rangle &=\lim\langle f_n(T^*T)Cg,h\rangle\\
	&=E_{Cg,h}(G)=\langle E(G)Cg,h\rangle.
	\end{split}
	\end{equation}
	This shows that $C$ commutes with spectral measure of $T^*T$.
	
	{\bf Proof of (\ref{spectral measure: CT}):} Since $T$ is $C$-normal, then we have
	\begin{equation}
	CT(T^*T)=(CTT^*)T=(T^*TC)T=T^*T(CT),
	\end{equation}
	which shows that $CT$ commutes with $T^*T$. Therefore, following the same steps as in (\ref{normal:C-normal:spectral measure}), we get the desired conclusion.
\end{proof}	
\begin{remark}
In a similar manner, we can show that	the anti-linear operator $TC$ commutes with the spectral measure of $TT^*$ as a consequence of the fact that $(TC)TT^*=T(CTT^*)=T(T^*TC)=TT^*(TC)$.
\end{remark}
	Before we proceed further, let us mention the following remarks which are the strong motivation behind our main results. 
\begin{remark}
	Let $(X, \mu)$ be a measure space. Let $L^2(X,\mu)$ be a space of complex valued functions with conjugation $S_\mu$ given by
	 \begin{equation}\label{conju on L^2}
		S_\mu f(x)=\overline{f(x)}.
	\end{equation}
	 Let $\psi\in L^\infty$ and $M_\psi$ be the multiplication operator on $L^2(X,\mu)$ defined by
	$M_\psi f=\psi f$. Then $M_\psi $ is $S_\mu$-symmetric, and hence $S_\mu$-normal.
	It is well known that from the spectral theory that any normal operator $N\in \clb(\clh)$ is unitary equivalent to the multiplication operator $M_\psi$ that is,
	$M_\psi=VNV^*$, where $V\in \clb(\clh,L^2(X,\mu))$ is a unitary. Let $\tilde{C}$ be a conjugation in $\clh$ such that $(V\tilde{C}V^*)f(x)=S_\mu f(x)$. Then $N$ is $\tilde{C}$-normal.	More precisely, if $N$ is a normal operator, then $N$ is $\tilde{C}$-normal, where the conjugation $\tilde{C}$ on $\clh$ is related to the conjugation $S_\mu$ on $L^2(X,\mu)$ by the relation \begin{equation}\label{second relation}
	\tilde{C}=V^*S_\mu V.
	\end{equation}	
\end{remark}	

\begin{remark}
	If $T\in\bh$ is normal and $C$-normal, then have seen that $CT^*T=T^*TC$. Since $T^*T$ is complex symmetric with respect to some conjugation, say $K$. Then by using \cref{anti-commut-normal}, $C=BK$ for some $B\in \bh$, where $B$ is $K$-symmetric unitary operator satisfying $BT^*T=T^*TB$.
\end{remark}
\begin{remark}
		If $T$ is $C$-normal, then we have observed that the anti-linear operator $CT$ commutes with the linear operator $T^*T$. Since $T^*T$ is complex symmetric with some conjugation, say $K$, by using \cref{anti-commut-normal}, we can say that $CT=DK$ for some $D\in \bh$ such that $D$ commute with $T^*T$. This implies that $T=CDK$ for some conjugation $K$. To be more precise, if $T\in \bh$ is $C$-normal, then it can be written as $T=CDK$ for some $D\in \bh$ and conjugation $K$. 
\end{remark}
This lead us to prove the following two results:

\begin{proposition}\label{Commute conju and T^*T}
Let $C$ be a conjugation on $\clh$, and suppose $T$ is $C$-normal. If $T$ is normal then 
\begin{equation}
C=U^*M_\vp V\tilde{C} V^*U,
\end{equation}
where $U,V:\clh\rightarrow L^2(\mu)$ are unitaries, $M_\vp$ is the multiplication operator on $L^2(\mu)$ and $\vp\in L^\infty(\mu)$.
\end{proposition}
\begin{proof}
	Since the linear operator $T^*T$ is positive, by spectral theorem for normal operator, there is a unitary operator $U:\clh\rightarrow L^2(\mu)$ such that $T^*T=U^*M_xU$, where $M_x$ is the multiplication operator induced by the real variable $x$. Therefore, $E_{T^*T}=U^*E_{M_x}U$, where $E_{T^*T}$ and $E_{M_x}$are the projection-valued spectral measures corresponding to $T^*T$ and $M_x$, respectively. Since $T$ is $C$-normal and normal, by (\ref{normal:C-normal:spectral measure}) of \cref{commute:spectral measure and conju}, we have
	\begin{equation}\label{antilinear commute spec measure}
	\begin{split}
	CU^*E_{M_x}U&=U^*E_{M_x}UC,\,\text{or}\\
	UCU^*E_{M_x} &=E_{M_x}UCU^*.
	\end{split}
	\end{equation}
	Observe that $S_\mu M_x=M_xS_\mu$, as well as $S_\mu E_{M_x}=E_{M_x}S_\mu$. Therefore, by \cref{antilinear commute spec measure}, we have
	\begin{equation*}
	UCU^*S_\mu E_{M_x} =E_{M_x}UCU^*S_\mu,
	\end{equation*}
	and this implies the linear operator $UCU^*S_\mu$ belongs to the commutant of $M_x$. 
	
Hence there exists an essentially bounded measurable function $\vp$ such that $UCU^*S_\mu=M_\vp$, where $M_\vp$ is the multiplication operator on $L^2(\mu)$ induced by $\vp$ with $|\vp|=1 $ $\mu$-a.e. It follows that 
	\begin{equation}\label{first relation}
	C=U^*M_\vp S_\mu U.
	\end{equation}
	Hence by using \cref{second relation}, we obtain the desired conclusion.
	
\end{proof}

\begin{proposition}\label{C-normal structure}
Let $C$ be a conjugation on $\clh$. Then $T$ is $C$-normal if and only if $T=CU^*M_\vp S_\mu U$ for some unitary $U\in \clb(\clh,L^2(\mu))$ and $\vp\in L^\infty(\mu)$.
\end{proposition}
\begin{proof}
Let $C$ be a conjugation on $\clh$. In (\ref{spectral measure: CT}) of \cref{commute:spectral measure and conju}, we have seen that if $T$ is $C$-normal, then the anti-linear operator $CT$ commutes with the projection-valued spectral measure $E_{T^*T}$ associated to the positive operator $T^*T$. By continuing, in the similar fashion as earlier, we have
\begin{equation}
\begin{split}
CTE_{T^*T}(G)&=E_{T^*T}(G)CT,\text{or}\\
CTU^*E_{M_x}U&=U^*E_{M_x}U CT,\,\text{that is,}\\
UCTU^*S_\mu E_{M_x}&=E_{M_x}U CT U^*S_\mu.
\end{split}
\end{equation}
Hence the linear operator $UCTU^*S_\mu$ is a commutant of $M_x$. Therefore, there exists $\vp\in L^\infty(\mu)$ such that $UCTU^*S_\mu=M_\vp$, which implies that $T=CU^*M_\vp S_\mu U$. 	
\end{proof}
The results in Propositions \ref{Commute conju and T^*T} and \ref{C-normal structure} are motivates to ask the following:

\begin{center}
	\begin{enumerate}
		\item 	\emph{Suppose that $T$ is $C$-normal and normal, then characterize $C$.}
		\item \emph{Suppose $T$ is $C$-normal, then what is the proper description or, structure of $T$?.}
	\end{enumerate}
\end{center}
Our primary objective in this paper is to provide a detailed answer to the points raised above. In the subsequent part of this section, we clarifies all.

\subsection{Factorization of $C$-normal operators}
In this subsection, we first characterize the conjugation $C$ whenever a $C$-normal operator is normal. For a given conjugation, it is evident that a $C$-normal operator $T$ is normal if and only if $C|T|=|T|C$, that is, the conjugation operator $C$ commute with the positive operator $|T|$. Therefore, it is enough to give the description of conjugation which commute with positive operator. In \cite{Gilbreath:Wogen}, Gillbreath and Wogen proved that a conjugation $J$ commuting with the positive operator $P$ if and only if $J=\left(\bigoplus_{n=1}^{\infty} J_{n}\right) \oplus J_{\infty}$ is a direct sum of conjugation operators and for each $n$, $ J_{n}=U_{n} S_{\mu_{n}}^{(n)}$, where $U_n$ is symmetric unitary operator matrix described as in \cref{conju commute positive} (up to unitary equivalence).

 Therefore, for a $C$-normal and normal operator $T$, by considering $P=|T|$, we have proved the following:
\begin{theorem}
Let $C$ be a conjugation, and let $T\in \bh$ be $C$-normal. Then $T$ is normal if and only if up to unitary equivalence $C=\left(\bigoplus_{n=1}^{\infty} C_{n}\right) \oplus C_{\infty}$ is a direct sum of conjugations operators and for each $n$, $ C_{n}=U_{n} S_{\mu_{n}}^{(n)}$, where $U_n$ is symmetric unitary operator matrix described as in \cref{conju commute positive}.
\end{theorem}

Next, we provide the structure of $C$-normal operators by giving the characterization of anti-unitary operators which commute with a positive operator. In order to do so, we deploy the method developed by Gillbreath and Wogen (cf. \cite{Gilbreath:Wogen}), where they have established the complete description of complex symmetric operators up to unitary equivalence by describing the conjugation that commutes with a positive operator. Also, they have shown that complex symmetric operators sit halfway between the normal operators and $\bh$.

 In \cite[Theorem 1.6]{Wang:Zhu}, Cun Wang, Jiayin Zhao, and Sen Zhu proved that $T\in \bh$ is $C$-normal if and only if $T=CJ|T|$, where $J$ is some anti-unitary operator which commute with $|T|$. Since our purpose is to settle down the same problem for $C$-normal operator, our initial approach is to give the description of anti-unitary operator which commute with a positive operator and this approach is quite intrinsic if we look at the polar decomposition for $C$-normal operators (cf. \cite[Theorem 1.6]{Wang:Zhu}).

  The following is the factorization theorem for $C$-normal operators.
\begin{theorem}
	Let $C$ be a conjugation on $\clh$. Then
	\[\cln_C(\clh)=\{CJP: P\,\,\text{is a positive operator and $J$ is an anti-unitary operator commuting with $P$} \}.\]
\end{theorem}

Here we aim to give the description of all anti-unitary operators which commute with the positive operator $P$. We shall use the spectral multiplicity theory to achieve our goal.

{\bf Case 1:} Suppose that $P$ has multiplicity one. There is a Borel measure $\mu$ whose support is the spectrum of $P$ such that $P\cong N_\mu$ ($P$ is unitarily equivalent to $N_\mu$), where $N_\mu f(t)=tf(t)$ on $L^2(\mu)$. Define the conjugation $S_\mu$ on $L^2(\mu)$ by $S_\mu f=\bar{f}$ for all $f\in L^2(\mu)$. Then it is very easy to see that $N_\mu$ is $S_\mu$-symmetric. The anti-unitary operator $J$ commute with the positive operator $N_\mu$. Therefore, by \cref{anti-commut-normal}, we have $J=US_\mu$, where $U$ is unitary satisfying $UN_\mu=N_\mu U$. Now the linear operator $U$ is unitary which commute with $P_\mu$, it follows that $U=M_\vp$ for some $\vp\in L^\infty(\mu)$ with $|\vp|=1$ a.e.. Hence we conclude that 
\[\{J: JN_\mu=N_\mu J\} =\{M_\vp S_\mu: \vp\in L^\infty(\mu)\,\,\text{and}\,\, |\vp|=1\,\text{a.e}\}. \]

{\bf Case 2:} Suppose that $P$ has uniform multiplicity $n$, $1< n\leq \infty$. Then there is a Borel measure $\mu$ such that $P\cong N^{(n)}_{\mu}$, where $N^{(n)}_{\mu}=\underbrace{N_\mu\oplus N_\mu\oplus \cdots\oplus N_\mu}_{\text{ $n$ copies}}.$ Then it is $S_\mu^{(n)}$ defines  conjugation on $L^2(\mu)^{(n)}=\underbrace{L^2(\mu)\oplus L^2(\mu)\oplus \cdots\oplus L^2(\mu)}_{\text{ $n$ copies}}$ and $N^{(n)}_{\mu}$ is $S_\mu^{(n)}$-symmetric. If $J$ is the anti-unitary operator commute with $N_\mu^{(n)}$, then by applying \cref{anti-commut-normal}, we get $J=US^{(n)}_\mu$, where $U=[U_{ij}]$ is a unitary operator on $L^2(\mu)^{(n)}$ such that $UN^{(n)}_\mu=N^{(n)}_\mu U$. Hence by using \cref{direct sum commut prop}, we get $U_{ij}=M_{\vp_{ij}}$, where $\vp_{ij}\in L^\infty(\mu)$ with $|\vp_{ij}|=1$ a.e. Thus, the anti-unitary operator $J$ commute with $N_\mu^{(n)}$ if and only if $J=[M_{\vp_{ij}}]S_\mu^{(n)}$, where $[M_{\vp_{ij}}]$ is an unitary operator matrix on $L^2(\mu)^{(n)}$.

 Therefore, by using \cref{Dsum_Normal}, from above discussion, we conclude that up to unitary equivalence,

\[P=\left(\bigoplus_{n=1}^{\infty} N_{\mu_{n}}^{(n)}\right) \oplus N_{\mu_{\infty}}^{(\infty)} \text { on } \mathcal{H}=\left(\bigoplus_{n=1}^{\infty} L^{2}\left(\mu_{n}\right)^{(n)}\right) \oplus L^{2} (\mu_{\infty})^{(\infty)},\]

and
\[\{P\}^{\prime}=\left(\bigoplus_{n=1}^{\infty}\left\{N_{\mu_{n}}^{(n)}\right\}^{\prime}\right) \oplus\left\{N_{\mu_{\infty}}^{(\infty)}\right\}^{\prime} .\]

Then $P$ is $S$-symmetric, where $S=\left(\bigoplus_{n=1}^{\infty} S_{\mu_{n}}^{(n)}\right) \oplus S_{\mu_{\infty}}^{(\infty)}$. Also, $J$ is an anti-unitary commuting with $P$ if and only if $J=\left(\bigoplus_{n=1}^{\infty} J_{n}\right) \oplus J_{\infty}$ is a direct sum of anti-unitary operators. For each $n$, $J_{n}=U_{n} S_{\mu_{n}}^{(n)}$, where $U_{n}$ unitary operator matrix. 

\begin{remark}
	\begin{enumerate}
		\item The anti-unitary operator $J_{n}=U_{n} S_{\mu_{n}}^{(n)}$, where $U_{n}$ unitary operator matrix is conjugation if and only if $U_{n}$ is symmetric unitary operator matrix.
		\item The anti-unitary operator $J_{n}=U_{n} S_{\mu_{n}}^{(n)}$, where $U_{n}$ unitary operator matrix is anti-conjugation if and only if $U_{n}$ is skew-symmetric unitary operator matrix.
	\end{enumerate}

\end{remark}

\section{Cartesian decomposition of $C$-normal operators revisited}
In this section, we shall present a different proof of \cite[Theorem 2.6]{RSV_Cart}, where Ramesh et al. have characterized $C$-normal operators by establishing the Cartesian decomposition. Here our proof is different and partially motivated from results related to the Cartesian decomposition of normal operators which is available in the book by V. S. Sundar (cf. \cite[p.~18]{Sundar_book}). From \cite[Theorem 2.3]{RSV_Cart}, we know that given a conjugation $C$, any bounded linear operator $T$ can be uniquely written as $T=A+iB$, where $A=\frac{T+CT^*C}{2}$ is $C$-symmetric and $B=\frac{T-CT^*C}{2i}$ is $C$-skew-symmetric.
\begin{theorem}\label{cart decomp C-normal}
	Let $C$ be a conjugation on $\clh$, and $T\in \clb(\clh)$ be such that $T=A+iB$, where $A$ is $C$-symmetric and $B$ is $C$-skew-symmetric. Then the following are equivalent:
	\begin{enumerate}
		\item\label{C-normal} $T$ is $C$-normal.
		\item\label{parallelogram law} $\left\|Tx\right\|^2=\left\|Ax\right\|^2+\left\|Bx\right\|^2$ for all $x\in \clh$.
		\item\label{adj_commute-1} $A^*B=B^*A$.
			\item\label{adj_commute-2} $AB^*=BA^*$.
	\end{enumerate}
In this case, $T^*T=A^*A+B^*B$ and $TT^*=AA^*+BB^*$ with $(A^*A)(B^*B)=(B^*B)(A^*A)$, and $\max \{\|A\|^2,\|B\|^2  \}\leq\|T\|^2\leq \|A\|^2+\|B\|^2$.
\end{theorem}
\begin{proof}
	Let $C$ be a conjugation, and $T=A+iB$, where $A$ is $C$-symmetric and $B$ is $C$-skew-symmetric (cf. \cite[Theorem 2.3]{RSV_Cart}). Therefore, for every $x\in \clh$
	\begin{equation}
	\begin{split}
	\left\|Tx\right\|^2=\left\|(A+iB)x\right\|^2=\left\|Ax\right\|^2+\left\|Bx\right\|^2-2\Re (i\langle Ax,Bx\rangle),
	\end{split}
	\end{equation}
	and 
	\begin{equation*}
	\begin{split}
	\left\|T^*x\right\|^2=\left\|(A^*-iB^*)x\right\|^2=\left\|A^*x\right\|^2+\left\|B^*x\right\|^2+2\Re (i\langle A^*x,B^*x\rangle).
	\end{split}
	\end{equation*}
	Since $T$ is $C$-normal, we have $\left\|Tx\right\|=\left\|T^*Cx\right\|$, and this implies 
	\begin{equation}
	\begin{split}
	\left\|Ax\right\|^2+\left\|Bx\right\|^2-2\Re (i\langle Ax,Bx\rangle)&=\left\|A^*Cx\right\|^2+\left\|B^*Cx\right\|^2+2\Re (i\langle A^*Cx,B^*Cx\rangle)\\
	&=\left\|Ax\right\|^2+\left\|Bx\right\|^2-2\Re (i\langle CAx,CBx\rangle)\\
	&=\left\|Ax\right\|^2+\left\|Bx\right\|^2-2\Re (i\langle Bx,Ax\rangle)\\
	&=\left\|Ax\right\|^2+\left\|Bx\right\|^2+2\Re (i\langle Ax,Bx\rangle)\quad(\Re z=\Re\bar{z}).
	\end{split}
	\end{equation}
	Thus for every $x\in \clh$, we have $T$ is $C$-normal if and only if $\left\|Tx\right\|=\left\|T^*Cx\right\|$ if and only if $\Re (i\langle Ax,Bx\rangle)=0$ if and only if  $\left\|Tx\right\|^2=\left\|Ax\right\|^2+\left\|Bx\right\|^2$. 
	
	Now observe that $\Re (i\langle Ax,Bx\rangle)=0$ if and only if $\langle Ax,Bx\rangle\in \mathbb{R}$ for all $x\in \clh$ if and only if $(B^*A)^*=B^*A$, that is, $A^*B=B^*A$. Since $A$ is $C$-symmetric and $B$ is $C$-skew-symmetric, we have
	$A^*B=B^*A\iff CA^*B=CB^*A \iff ACB=-BCA \iff -AB^*C=-BA^*C\iff AB^*=BA^*$. This completes the proof.
\end{proof}	
	
We characterize $C$-normal finite-dimensional truncated weighted shifts as an application of Cartesian decomposition of $C$-normal operators.
\begin{corollary}\label{cwshift}
	Let $C$ be the conjugation on $\mathbb{C}^n$ defined by 
	\begin{align}\label{standconju}
		C(z_1,z_2,\dots, z_n)=(\widebar{z_n},\widebar{z_{n-1}},\dots, \widebar{z_2},\widebar{z_1}).
	\end{align} Then
	\begin{align}\label{wshft}
		T=\displaystyle\sum_{j=1}^{n-1}\lambda_j e_j\otimes e_{j+1}
	\end{align} is $C$-normal if and only if $|\lambda_j|=|\lambda_{n-j}|$ for $j=1,2,\dots, n-1$.
\end{corollary}
\begin{proof}
	To prove this, we use the Cartesian decomposition of $C$-normal operators (See Theorem \ref{cart decomp C-normal}). We write $T=A+iB$, where
	\begin{align*}
		A=\frac{T+CT^*C}{2}=\frac{1}{2}\displaystyle\sum_{j=1}^{n-1}(\lambda_{n-j}+\lambda_j)e_j\otimes e_{j+1}
	\end{align*}
	and 
	\begin{align*}
		B=\frac{T-CT^*C}{2i}=\frac{1}{2i}\displaystyle\sum_{j=1}^{n-1}(\lambda_j-\lambda_{n-j})e_j\otimes e_{j+1}.
	\end{align*} 
	Also note that
	$$AB^*=\displaystyle\sum_{j=1}^{n-1}\frac{(\lambda_{n-j}+\lambda_j)(\bar{\lambda_{n-j}}-\bar{\lambda_j})}{4i}e_j\otimes e_j$$ and
	$$BA^*=\displaystyle\sum_{j=1}^{n-1}\frac{(\lambda_j-\lambda_{n-j})(\bar{\lambda_{n-j}}+\bar{\lambda_j})}{4i}e_j\otimes e_j.$$
	The condition $AB^*=BA^*$ implies
	\begin{align*}
		(\lambda_{n-j}+\lambda_j)(\bar{\lambda_{n-j}}-\bar{\lambda_j})=(\lambda_j-\lambda_{n-j})(\bar{\lambda_{n-j}}+\bar{\lambda_j})
	\end{align*} and this implies $|\lambda_{n-j}|=|\lambda_j|$ for all $1\leq j\leq n-1$. Hence the conclusion follows.
\end{proof}

\begin{example}
	Let $C$ be  conjugation defined by  $Cf(t)=\overline{f(1-t)}$ for all $t\in [0,1]$ and consider the multiplication operator $M_\vp f=\vp f$ for all $f\in L^2[0,1]$. Then $M_\vp\in \clb(L^2[0,1])$ for $\vp\in L^\infty$. Then \begin{equation*}
		A=\frac{1}{2}M_{\psi_{1}}\quad\text{and}\quad 	B=\frac{1}{2i}M_{\psi_{2}},
	\end{equation*}
	where $\psi_1(t)=\vp(t)+\vp(1-t)$ and $\psi_2(t)=\vp(t)-\vp(1-t)$. Therefore, by Theorem \ref{cart decomp C-normal}, $M_\vp$ is $C$-normal if and only if $A^*B=B^*A$ that is $\bar{\psi_{1}}\psi_2=\bar{\psi_{2}}\psi_1$ which implies $|\vp(1-t)|^2=|\vp(t)|^2$.
\end{example}	
	
	Next, we shall give an application of Cartesian decomposition of $C$-normal operators.
\subsection{Singular value inequalities for compact $C$-normal operators}

Let $T\in \bh$ be a compact operator. Then the singular
values of $T$, denoted by $s_1(T),s_2(T),\dots$ are the eigenvalues of the positive operator $|T|=(T^*T)^{\frac{1}{2}}$
enumerated as $s_1(T)\geq s_2(T)\geq \dots$, and repeated according to multiplicity. Note that $s_j(T) = s_j(T^*) =s_j(|T|)$ for $j = 1, 2,\dots $. It follows by Weyl's monotonicity principle (cf. \cite[p.~63]{Bhatia}) that
if $S,T$ are positive compact operators and $S \leq T$ (that is, $T-S$ is a positive operator), then $s_j(S) \leq s_j(T)$ for $j = 1, 2,\dots$.

\begin{theorem}\label{singular value ineq}
Let $C$ be a conjugation, and $T$ be a compact $C$-normal operator, where $T=A+iB$ be the Cartesian decomposition of $T$. Then 
\[ \frac{1}{\sqrt{2}}s_j(|A|+|B|)\leq s_j(T)\leq s_j(|A|+|B|),  \]
for $j=1,2,\dots$.
\end{theorem}
\begin{proof}
Since $T$ is a compact $C$-normal operator, it follows from \cref{cart decomp C-normal} that $T=A+iB$, where $A$ is $C$-symmetric and $B$ is $C$-skew-symmetric such that $A^*B=B^*A$ and $|T|=\sqrt{|A|^2+|B|^2}$ with $|A||B|=|B||A|$. Thus we have 
\begin{equation}
\frac{1}{\sqrt{2}}(||A|+|B||)\leq \sqrt{|A|^2+|B|^2}=|T|\leq |A|+|B|.
\end{equation}
Hence, by using Weyl's monotonicity principle the conclusion follows. 
\end{proof}
\begin{remark}
	The equality holds in the right hand side of the inequality in \cref{singular value ineq} if either $A=0$, or $B=0$.
\end{remark}
\begin{corollary}
Let $T\in \bh$ be a compact $C$-normal with the Cartesian decomposition $T=A+iB$. Then $2s_j(BA^*)=2s_j(AB^*)\leq s_j((T^*T)\oplus(TT^*) )$ for all $j=1,2,\dots$.
\end{corollary}
\begin{proof}
	 It is well known that if $A,B$ are compact operators, then
	\begin{equation}\label{Bhati-Kitta ineq}
	s_j(AB^*+BA^*)\leq s_j((A^*A+B^*B)\oplus (AA^*+BB^*)),
	\end{equation}
	for $j=1,2,\dots$ (cf. \cite{Bhatia:Kittaneh}).
	Therefore, if $T\in \bh$ is compact $C$-normal with the Cartesian decomposition $T=A+iB$, by using \cref{cart decomp C-normal}, we have $T^*T=A^*A+B^*B$ and $TT^*=AA^*+BB^*$, and hence by \cref{Bhati-Kitta ineq}, we have $2s_j(BA^*)=2s_j(AB^*)\leq s_j((T^*T)\oplus(TT^*) )$ for $j=1,2,\dots$.
\end{proof}

\begin{lemma}
	Let $C$ be a conjugation, and $A\in \bh$ such that $A^2-{A^*}^2$ is $C$-skew-symmetric. Then $T=A+iA^*$ is $C$-normal if and only if $A$ is $C$-normal.
\end{lemma}
\begin{proof}
	The desired result follows easily from a direct calculation. 
\end{proof}
\begin{remark}
 $T\in \bh$ is conjugate normal if and only if $iT$ is conjugate normal.
\end{remark}

\subsection{Some inequalities for self-commutator}

Let $T\in \bh$ be a $C$-normal operator with the Cartesian decomposition $T=A+iB$ as in \cref{cart decomp C-normal}. Then the following are true:
\begin{enumerate}
	\item $\|T^*T-TT^*\|\leq 2\|A\|\min\{\|A-A^*\|,\|A+A^*\|  \}+2\|B\|\min\{\|B-B^*\|,\|B+B^*\|  \}$. We get it by using \cite[Corollary 1]{Dragomir}.
	\item If $\alpha,\beta,\gamma,\delta\in \C$ such that $C_{\alpha,\beta}(A)$ and $C_{\gamma,\delta}(B)$ are accretive (cf. \cite[Section 3]{Dragomir-LAA}), then  $\|T^*T-TT^*\|\leq \frac{1}{2}|\alpha-\beta|^2+\frac{1}{2}|\gamma-\delta|^2$. We get it by using \cite[Theorem 6]{Dragomir}.
\end{enumerate}

\begin{center}
	\textbf{Acknowledgements}
\end{center}
The research of the author is supported by the NBHM postdoctoral fellowship, Department of Atomic Energy (DAE), Government of India (File No: 0204/16(21)/2022/R\&D-II/11995).

\end{document}